\documentclass[10pt]{amsart}
\usepackage{amsmath,amssymb,graphicx}

\begin{document}

\newtheorem{thm}{Theorem}[section]
\newtheorem{lem}[thm]{Lemma}
\newtheorem{prop}[thm]{Proposition}
\newtheorem{cor}[thm]{Corollary}
\newtheorem{conj}[thm]{Conjecture}
\newtheorem{defn}[thm]{Definition}
\newtheorem{remark}{Remark}

\numberwithin{equation}{section}

\newcommand{\Z}{{\mathbb Z}} %cph changed from \mathbf
\newcommand{\Q}{{\mathbb Q}}
\newcommand{\R}{{\mathbb R}}
\newcommand{\C}{{\mathbb C}}
\newcommand{\N}{{\mathbb N}}
\newcommand{\FF}{{\mathbb F}}
\newcommand{\T}{{\mathbb T}}
\newcommand{\fq}{\mathbb{F}_q}
\newcommand{\fixme}[1]{\footnote{Fixme: #1}}

\def\scrA{{\mathcal A}}
\def\scrB{{\mathcal B}}
\def\scrD{{\mathcal D}}
\def\scrE{{\mathcal E}}
\def\scrH{{\mathcal H}}
\def\scrK{{\mathcal K}}
\def\scrL{{\mathcal L}}
\def\scrM{{\mathcal M}}
\def\scrN{{\mathcal N}}
\def\scrQ{{\mathcal Q}}
\def\scrS{{\mathcal S}}

\newcommand{\rmk}[1]{\footnote{{\bf Comment:} #1}}

\renewcommand{\mod}{\;\operatorname{mod}}
\newcommand{\ord}{\operatorname{ord}}
\newcommand{\TT}{\mathbb{T}}
\renewcommand{\i}{{\mathrm{i}}}
\renewcommand{\d}{{\mathrm{d}}}
\renewcommand{\^}{\widehat}
\newcommand{\HH}{\mathbb H}
\newcommand{\Vol}{\operatorname{vol}}
\newcommand{\area}{\operatorname{area}}
\newcommand{\tr}{\operatorname{tr}}
\newcommand{\norm}{\mathcal N} % norm =(\frac{ n+\sqrt{n^2-4}} 2)^2
\newcommand{\fnorm}[1]{\left\lVert #1 \right\rVert}
\newcommand{\intinf}{\int_{-\infty}^\infty}
\newcommand{\ave}[1]{\left\langle#1\right\rangle} %  average
\newcommand{\Var}{\operatorname{Var}}
\newcommand{\Cov}{\operatorname{Cov}}
\newcommand{\Prob}{\operatorname{Prob}}
\newcommand{\sym}{\operatorname{Sym}}
\newcommand{\disc}{\operatorname{disc}}
\newcommand{\CA}{{\mathcal C}_A}
\newcommand{\cond}{\operatorname{cond}} % conductor
\newcommand{\lcm}{\operatorname{lcm}}
\newcommand{\Kl}{\operatorname{Kl}} %Kloosterman sum
\newcommand{\leg}[2]{\left( \frac{#1}{#2} \right)}  % Legendre symbol
\newcommand{\SL}{\operatorname{SL}}

\newcommand{\be}{\begin{equation}}
\newcommand{\ee}{\end{equation}}
\newcommand{\bs}{\begin{split}}
\newcommand{\es}{\end{split}}
\newcommand{\bra}{\left\langle}
\newcommand{\ket}{\right\rangle}

\newcommand{\sumstar}{\sideset \and^{*} \to \sum}

\newcommand{\LL}{\mathcal L} %L-function of u
\newcommand{\sumf}{\sum^\flat}
\newcommand{\Hgev}{\mathcal H_{2g+2,q}}
\newcommand{\USp}{\operatorname{USp}}
\newcommand{\conv}{*}
\newcommand{\dist} {\operatorname{dist}}
\newcommand{\CF}{c_0} % Fejer constant
\newcommand{\kerp}{\mathcal K}

\newcommand{\gp}{\operatorname{gp}}
\newcommand{\Area}{\operatorname{Area}}

\newcommand{\Op}{\operatorname{Op}}
\newcommand{\Dom}{\operatorname{Dom}}

\title[Quantum Ergodicity for toral point scatterers]{Quantum
  Ergodicity for Point Scatterers on  Arithmetic Tori} 
\author{P\"ar Kurlberg}
\address{Department of Mathematics, KTH Royal Institute of Technology, SE-10044 \\ Stockholm, Sweden}
\email{kurlberg@math.kth.se}
\author{Henrik Uebersch\"ar}
\address{Institut de Physique Th\'eorique, CEA Saclay, 91191 Gif-sur-Yvette Cedex, France.}
\email{henrik.uberschar@cea.fr}
%\date{\today}
\date{\today}

\thanks{P.K. was partially supported by grants from 
the G\"oran Gustafsson Foundation
and the Swedish Research Council.}

\begin{abstract}
  We prove an analogue of Shnirelman, Zelditch and Colin de Verdi\`e-
  re's Quantum Ergodicity Theorems in a case where there is {\em no}
  underlying classical ergodicity. The system we consider is the
  Laplacian with a delta potential on the square torus.  There are two
  types of wave functions: old eigenfunctions of the Laplacian, which
  are not affected by the scatterer, and new eigenfunctions which have
  a logarithmic singularity at the position of the scatterer. We prove
  that a full density subsequence of the new eigenfunctions {\em
    equidistribute in phase space}. Our estimates are uniform with respect
    to the coupling parameter, in particular the equidistribution holds
    for both the weak and strong coupling quantizations of the point scatterer.
\end{abstract}

\maketitle

\section{Introduction}
The point scatterer, namely the Laplacian with a delta potential, on a
two-dimensional flat manifold is a popular model in the study of the
transition between chaos and integrability in quantum systems. In 1990
Seba \cite{Seba} considered this operator on a rectangle with
irrational aspect ratio and Dirichlet boundary conditions and argued
that the spectrum and eigenfunctions of the point scatterer display
features such as level repulsion and a Gaussian value distribution,
both of which are present in quantum systems with chaotic 
classical dynamics (cf. \cite{CdV2} and \cite{BohigasGiannoniSchmit}),
such as the quantization of the geodesic flow on hyperbolic manifolds
or the flow in the Sinai billiard. In fact the point scatterer can be understood as
a limit of the Sinai billiard where the radius shrinks to zero faster than the
semiclassical wavelength.

The subject of this paper is a point scatterer on a flat torus. It has two types of eigenfunctions: 
%\begin{itemize}
%\item[] 
first there are {\em old} eigenfunctions of the Laplacian, namely
those which vanish at the position of the scatterer; the nonzero eigenvalues
remain the same, though with multiplicities reduced by $1$.
%\item[] 
Secondly, there are {\em new} eigenfunctions which diverge
logarithmically near the position of the scatterer; the corresponding
eigenvalues have multiplicity $1$ and interlace with the old Laplace
eigenvalues.
%\end{itemize}

We shall  only be concerned with the set of new
eigenfunctions, i.e., the ones which are affected by the scatterer.  In
\cite{RU} it was proved that a full density subsequence\footnote{See
  Section~\ref{PureModes} for a precise definition of a ``full density
  subsequence''.} of the new eigenfunctions {\em equidistribute in
  position} space in the special case of a square torus. We extend the results of \cite{RU} and prove that a
full density subsequence of these eigenfunctions in fact {\em
  equidistribute in phase space} 
--- we thus establish an analogue of Shnirelman, Zelditch and Colin de
Verdi\`ere's Quantum Ergodicity Theorem in a case where there is no
underlying chaotic dynamics and no classical ergodicity.

An analogue of this result for a cubic 3D torus was recently obtained
by N. Yesha \cite{Yesha2}. The situation for a square torus is very
different from irrational tori, where the eigenfunctions are expected
to localise in phase space on a finite number of momentum vectors \cite{KeatingMarklofWinn2,BerkolaikoKeatingWinn2}. In
the case where the aspect ratio is diophantine this can be proven
rigorously for a full density subsequence of new eigenfunctions
\cite{KU2}.

\subsection{Spectrum of the point scatterer}
The formal operator 
% in general no sign condition here, only the weak coupling regime corresponds to alpha negative%
$$-\Delta+\alpha\delta_{x_0}, \quad \alpha\in\R$$
is realized using von Neumann's theory of self-adjoint extensions. We
simply state the most important facts in this section in order to
formulate the results of this paper. For a more detailed discussion of
the self-adjoint realization of the point scatterer we refer the
reader to the introduction and appendix of the paper \cite{RU}.

Let $\T^2=\R^2/2\pi\Z^2$. We consider the restriction of the positive
Laplacian $-\Delta$ to the
domain $$D_0=C^\infty_c(\T^2\setminus\{x_0\})$$ of functions which 
% I replaced H_0 by H.
vanish near the position of the scatterer: $$H=-\Delta|_{D_0}$$ The
operator $H$ is symmetric, but fails to be self-adjoint, in fact
$H$ has deficiency indices $(1,1)$. Self-adjoint extension theory
tells us that there exists a one-parameter family of self-adjoint
extensions $H_\varphi$, $\varphi\in(-\pi,\pi]$, which are restrictions
of the adjoint $H^*$ to the domain of functions $f\in \Dom(H^*)$
which satisfy the asymptotic
$$f(x)=C\left(\cos\left(\frac{\varphi}{2}\right)\frac{\log|x-x_0|}{2\pi}+\sin\left(\frac{\varphi}{2}\right)\right)+o(1), \quad x\to x_0$$
for some constant $C\in\C$. The case $\varphi=\pi$
corresponds to $\alpha=0$. In this paper we will study the operators
$H_\varphi$, $\varphi\in(-\pi,\pi)$.

The spectrum of the operator $H_\varphi$ consists of two parts:
``old'' and``new'' eigenvalues. Since $H_\varphi$ is a self-adjoint
realization of a rank one perturbation of the Laplacian, the effect is
that each nonzero old Laplace eigenvalue appears, with multiplicity reduced by
$1$, in the spectrum of $H_\varphi$. Further, each old Laplace
eigenvalue gives rise to a new eigenvalue with multiplicity $1$.  In
fact, these new eigenvalues {\em interlace} with the multiplicity one
sequence associated with the old Laplace eigenvalues.

There are two types of eigenfunctions of $H_\varphi$ associated with
the two parts of the spectrum:
\begin{itemize}
\item[(A)] ``Old'' eigenfunctions which vanish at $x_0$ and therefore
  are not affected by the scatterer. These are simply eigenfunctions
  of
  the unperturbed Laplacian.
\item[(B)] ``New'' eigenfunctions which feature a logarithmic
  singularity at $x_0$; in fact they are given by Green's functions
  $G_\lambda=(\Delta+\lambda)^{-1}\delta_{x_0}$.

\end{itemize} 

We will study how eigenfunctions of type (B) are distributed in phase
space as the eigenvalue tends to infinity. Denote by $S$ the set of
distinct eigenvalues of the Laplacian on $\T^2$, namely integers which
can be represented as a sum of two squares:
$$S:=\{n \in \Z : n =x^2+y^2 \mid x,y\in\Z\}$$
For given $n\in S$ denote its multiplicity by 
$$r_2(n):=\sum_{\substack{n=|\xi|^2 \\ \xi\in\Z^2}}1,$$ 
i.e. the number of ways $n$ can be
written as a sum of two squares.  

The eigenvalues of type (B) are solutions to the equation
\begin{figure}\label{figure}
\centering
\includegraphics[scale=0.65]{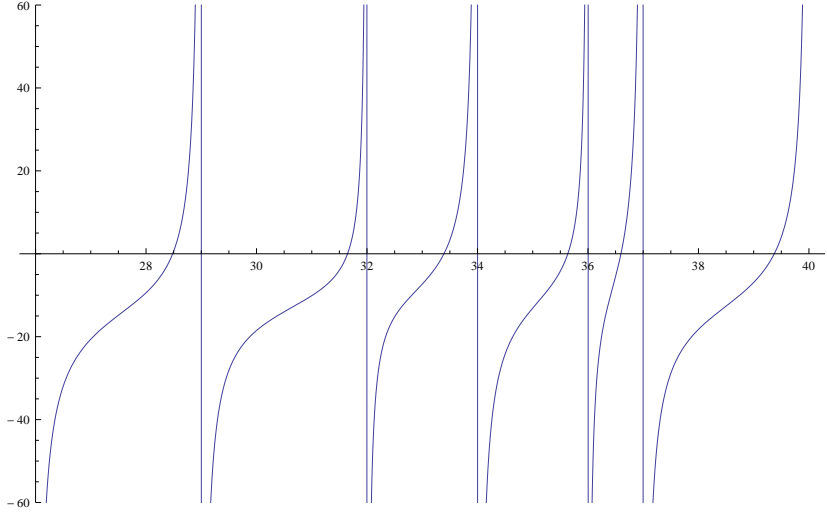}
\caption{The picture shows a plot of the l. h. s. of equation \eqref{weak coupling quantization} as a function of $\lambda$. The zeroes are the new eigenvalues corresponding to the self-adjoint extension with parameter $\varphi=0$.}
\end{figure}
\be\label{weak coupling quantization}
\sum_{n\in S}r_2(n)\left(\frac{1}{n-\lambda}-\frac{n}{n^2+1}\right)=c_0\tan\left(\frac{\varphi}{2}\right)
\ee
(see Figure~1 for a plot of the l. h. s.) where 
\be
c_0=\sum_{n\in S}\frac{r_2(n)}{n^2+1}.
\ee
As mentioned earlier, they interlace with the distinct Laplace
eigenvalues $$S=\{0<1<2<4<5<8<\cdots\}$$ as follows 
\be
\lambda_{0,\varphi}<0<\lambda_{1,\varphi}<1<\lambda_{2,\varphi}<2<\lambda_{4,\varphi}<4<\lambda_{5,\varphi}<5<\lambda_{8,\varphi}<8<\cdots
\ee
where the new eigenvalue associated with $n\in S$ is denoted by
$\lambda_{n,\varphi}$; note that $\lambda_{n,\varphi} < n$.

\subsection{Strong coupling}
In the physics literature equation \eqref{weak coupling quantization}
is referred to as a ``weak coupling'' quantization. In fact, (cf. \cite{RU2})
the new eigenvalues $\lambda_{m,\varphi}$ ``clump'' with the Laplace
eigenvalues $m\in S$ in the sense that for a full density subsequence
of S,
$$ 0 <m-\lambda_{m,\varphi} \ll \frac{1}{(\log  m)^{1-o(1)}}.$$ 
In particular the eigenvalue spacing distribution of the point
scatterer coincides with that of the Laplacian, and the effect of the
scatterer on the spectrum is quite weak in this quantization.  (In
a sense it corresponds to letting $\alpha \to 0$ as $\lambda
\to\infty$.)

Shigehara \cite{Shigehara1} and later Bogomolny, Gerland and Schmit 
\cite{BogomolnyGerlandSchmit}, with the intent of finding a model
exhibiting level repulsion, 
considered 
another 
quantization
sometimes referred to as a ``strong coupling'' quantization. 
There are various ways to arrive at
this quantization condition from equation \eqref{weak coupling
  quantization}. 
For example, one may truncate the summation outside an energy window
of size $O(\lambda^\delta)$ where $\delta>0$ is fixed, and the new
eigenvalues
of the strong coupling quantization are then defined to be the
solutions to the equation \be \label{strong coupling quantization}
\sum_{\substack{n\in S \\ |n-n_+(\lambda)|\leq
    n_+(\lambda)^\delta}}r_2(n)\left(\frac{1}{n-\lambda}-\frac{n}{n^2+1}\right)=c_0\tan\left(\frac{\varphi}{2}\right),
\ee where $n_+(\lambda)$ denotes the smallest element of $S$ which is
larger than $\lambda$.  (With $\lambda$ denoting such a
solution, the corresponding ``new'' eigenfunction is defined as a
certain Green's function $G_{\lambda}$,
cf. Section~\ref{sec:semicl-meas}.) 
A  summation by parts argument (see for instance Lemma 3.1 in
\cite{U2}) shows that
$$
\sum_{\substack{n\in S \\ |n-n_+(\lambda)|>
    n_+(\lambda)^\delta}}r_2(n)\left(\frac{1}{n-\lambda}-\frac{n}{n^2+1}\right)
= -\pi \log \lambda + O_{\delta}(1)
$$
and hence the truncation given by \eqref{strong coupling quantization}
is equivalent to a logarithmic    
renormalisation of the r.~h.~s. of \eqref{weak coupling quantization},
namely, as $\lambda \to \infty$, \be
\label{renormalization}
\sum_{n\in
  S}r_2(n)\left(\frac{1}{n-\lambda}-\frac{n}{n^2+1}\right)=
-\pi(1+o_\delta(1))\log\lambda =
c_0\tan\left(\frac{\varphi_\lambda}{2}\right), \ee if we allow 
$\varphi_{\lambda}$ to depend on $\lambda$ appropriately, and where the
$o_\delta(1)$ error term depends on the exponent $\delta$.
Since the error term depends on $\delta$ we note that there is no
unique choice of strong coupling quantization; the key point is
matching the leading order logarithmic term.
The renormalization in \eqref{renormalization} can be viewed as
letting a boundary condition vary with the energy.  Consequently
$D_{\varphi_{\lambda}}$, the domain of the operator
$H_{\varphi_{\lambda}}$ is varying; this setting is reminiscent of
problems in semiclassical analysis where boundary conditions are
allowed to depend on the semiclassical parameter $\hbar$.

\begin{remark}
  In the weak coupling quantization the lowest new eigenvalue is
  always negative, but for the strong coupling quantization the
  lowest new eigenvalue may be either positive or negative. In the
  case of a positive lowest new eigenvalue, this eigenvalue would be
  denoted $\lambda_1$ to keep our notation consistent, in particular
  ensuring that $\lambda_n<n$ for $n \in S$ always holds.
\end{remark}

%\begin{remark}
We remark that in the statement of our main result, Theorem \ref{QE},
the sequence $\Lambda=\{\lambda_n\}$ will denote {\em {\bf any}
  increasing sequence of numbers which interlace with $S$}.  In
particular, it applies to the eigenvalues of the weak, as well as the
strong, coupling quantizations.
%\end{remark}

\subsection{Semiclassical Measures}
\label{sec:semicl-meas}
Let $a\in C^\infty(S^*\T^2)$. Denote by $\Op(a)$ a zero-order
pseudo-differential operator associated with
$a$ (see subsection \ref{PseudoDiffCalc} for more details.)

Let $g_\lambda=G_\lambda/\|G_\lambda\|_2$,
$\lambda\notin S$, where we recall that $S$ denotes the set of Laplace
eigenvalues and that $G_\lambda=(\Delta+\lambda)^{-1}\delta_{x_0}$. We are
interested in 
weak limits of measures $d\mu_\lambda$ defined by the identity
\be\label{semimeas} \bra \Op(a)g_\lambda,g_\lambda\ket=\int_{S^*\T^2} a
d\mu_\lambda.  \ee

\subsection{Main Result}
The following theorem holds generally for the $L^2$-normalized Green's
functions $g_\lambda$.  It states that the measures $d\mu_\lambda$
defined by \eqref{semimeas} converge weakly to Liouville measure as
$\lambda \to \infty$ along a full density subsequence of {\em any}
increasing sequence $\Lambda$ which interlaces with $S$.  (Recall that $S$
denotes set of unperturbed Laplace eigenvalues, namely the set of
integers which can be represented as a sum of two squares.)

\begin{thm}\label{QE}
  Let $\Lambda$ be an increasing sequence which interlaces with $S$. For $m
  \in S$, denote
  by $\lambda_m$ the largest element of $\Lambda$ which is smaller than
  $m\in S$. There exists a full density subsequence $S'\subset S$,
  that does not depend on $\Lambda$, such that for all $a\in
  C^\infty(S^*\T^2)$, \be \lim_{\substack{m\to\infty \\ m\in S'}}\bra
  \Op(a)g_{\lambda_m},g_{\lambda_m}\ket = \int_{S^*\T^2}
  a(x,\varphi)\frac{dx\;d\varphi}{\Vol(S^*\T^2)}.  \ee
\end{thm}

As already noted,  the theorem holds in particular
for the new eigenvalues of the weak and strong coupling
quantizations of a point scatterer. Hence we have the following
corollary of Theorem~\ref{QE}. 
\begin{cor}
Quantum Ergodicity holds for the new eigenfunctions of weakly, as
well as strongly, coupled point scatterers on $\T^2$. 
\end{cor}
\begin{remark}
  Recall that the new eigenvalues in the strong coupling limit are
  given by the set of solutions $\{ \lambda_{m}\}_{m}$ to
  \eqref{renormalization} (or alternatively, solutions to
  \eqref{strong coupling quantization}), with corresponding new
  eigenfunctions given by the Green's functions $G_{\lambda_{m}}$.
  Although these Green's functions are eigenfunctions of different
  operators $\{ H_{\varphi_{\lambda_{m}}}\}_{m}$ (in fact, the domains
  of the operators change), it is natural to say that quantum
  ergodicity holds in the strong coupling limit if a full density
  subset of the collection of new eigenfunctions equidistribute.
\end{remark}

We further note that the counting function, or Weyl's law, for the set
of new eigenvalues (cf. Theorem~\ref{thm:landau}) satisfies
$$|\{ n : \lambda_n \leq x\}|
\ll \frac{x}{ \sqrt{\log x}}  = o(x),
$$
while the counting function for the full set of eigenvalues (new and
old, with multiplicity) is the same as for the unperturbed Laplacian,
hence $ \gg x$.  Consequently, the sequence of new eigenvalues is of
density zero within the full spectrum, and the approach of proving
Quantum Ergodicity for the set of new eigenfunctions by computing
first or second moments of matrix coefficients (e.g., see
\cite{Zelditch}) with respect to the full set of eigenfunctions seems
unlikely to succeed.

\subsection{Acknowledgements}
\label{sec:acknowledgements}
We would like to thank Zeev Rudnick and Stephane Nonnenmacher for
valuable discussions about this problem and for many helpful remarks
which have led to the improvement of this paper.  

The authors are also very grateful to the referee for a careful reading of
the paper and for many comments and suggestions that improved the
exposition.

\section{The matrix elements}

\subsection{Quantization of phase space observables}\label{PseudoDiffCalc}
Consider a classical symbol $a\in C^\infty(S^*\T^2)$, where
$S^*\T^2\simeq \T^2\times S^1$ denotes the unit cotangent bundle of
$\T^2$. We may expand $a$ in the Fourier series 

(note that the Fourier coefficients
decay rapidly since $a$ is smooth)
\begin{equation}
a(x,\phi)=\sum_{\zeta\in\Z^2,k\in\Z}\hat{a}(\zeta,k)e^{\i \left\langle \zeta,x \right\rangle+\i k\phi}.
\end{equation}

We choose a complex realization of the unit cotangent bundle $S^*\T^2$
and parametrise the unit circle $S^1$ at position $x\in\T^2$ by the
complex exponential map $\varphi \mapsto e^{\i\varphi}$.
We now want to associate with $a$ a pseudodifferential
operator
$\Op(a):C^\infty(\T^2)\to C^\infty(\T^2)$.  We choose
the following symbol (we associate with $\xi=(\xi_1,\xi_2)$ the complex number
$\tilde{\xi}:=\xi_1+\i\xi_2$ and note that $e^{\i k\arg\tilde{\xi}}=(\tilde\xi/|\tilde\xi|)^k$)
\be \sigma_{a}(x,\xi)=
\begin{cases}
\sum_{\zeta\in\Z^2, k\in\Z}
\hat{a}(\zeta,k)\left(\frac{\tilde{\xi}}{|\tilde{\xi}|}\right)^k
e^{\i\left\langle \zeta,x \right\rangle}, \quad \xi\neq0\\ 
\\
\sum_{\zeta\in\Z^2, k\in\Z} \hat{a}(\zeta,k)e^{\i\left\langle \zeta,x \right\rangle}, \quad \xi=0.
\end{cases}
\ee

Claim: The symbol $\sigma_{a}$, as defined above, belongs to the class
of toroidal symbols $S^{0}_{1,0}(\T^2\times\Z^2)$ as defined in
\cite{RT}, Part II, Section 4.1.2, Defn. 4.1.7, p. 344.
To see this, define the difference operators $$\Delta_{\xi_j}
f(\xi)=f(\xi+e_j)-f(\xi),$$ where $e_1=(1,0)$, $e_2=(0,1)$. By the
mean value theorem for repeated differences, for $f : \R^2 \to \R$ a
smooth function,
$$
\Delta_{\xi_1}^{\beta_{1}} \Delta_{\xi_2}^{\beta_{2}}
f(\xi)=
\partial_{\xi_1}^{\beta_{1}}
\partial_{\xi_2}^{\beta_{2}} f( \xi) \Big|_{\xi=\xi'}, 
$$
for some $\xi'=\xi+(\beta_{1}',\beta_{2}')$ with
$(\beta_{1}',\beta_{2}') \in [0,\beta_1] \times [0,\beta_{2}] $. 
With $f_{k}(\xi)$ denoting the real, or imaginary, part of $(
  \tilde{\xi}/|\tilde{\xi}|)^k$, a quick calculation then gives
that for integers $\alpha_1,\alpha_{2},\beta_1,\beta_{2} \geq 0$,
$$
\left| \partial_{x_1}^{\alpha_1}\partial_{x_2}^{\alpha_2}
\Delta_{\xi_1}^{\beta_1}\Delta_{\xi_2}^{\beta_2} 
f_{k}(\xi)
e^{\i\left\langle \zeta,x \right\rangle} \right|
\ll_{\alpha_1,\alpha_{2},\beta_1,\beta_{2}}
k^{\beta_1+\beta_2}
|\zeta_1|^{\alpha_{1}} |\zeta_2|^{\alpha_{2}}
(1+|\xi|)^{-\beta_1-\beta_2}
$$
This bound, together with the rapid decay of Fourier coefficients of
$a$, implies that
$$|\partial_{x_1}^{\alpha_1}\partial_{x_2}^{\alpha_2}
\Delta_{\xi_1}^{\beta_1}\Delta_{\xi_2}^{\beta_2}\sigma_{a}(x,\xi)|
\leq C_{a,\alpha_1,\alpha_2,\beta_1,\beta_2}
(1+|\xi|)^{-\beta_1-\beta_2},
$$ 
thus confirming the claim.

The action of the pseudodifferential operator $\Op(a)$ is then defined by multiplication on the Fourier side, analogously to Defn. 4.1.9 in \cite{RT}. Therefore, we have
\be 
\begin{split}
(\Op(a)f)(x)=&\sum_{\xi\in\Z^2}\sigma_{a}(x,\xi)\hat{f}(\xi)e^{\i\left\langle\xi,x\right\rangle}\\
=&\sum_{\xi\in\Z^2\setminus\{0\}}\sum_{\zeta\in\Z^2, k\in\Z} \hat{a}(\zeta,k)\left(\frac{\tilde{\xi}}{|\tilde{\xi}|}\right)^k 
\hat{f}(\xi) e^{\i\left\langle\xi+\zeta,x\right\rangle}\\
&+\sum_{\zeta\in\Z^2, k\in\Z} \hat{a}(\zeta,k)\hat{f}(0)e^{\i\left\langle \zeta,x \right\rangle}
\end{split}
\ee

We can now read off the action of $\Op(a)$ on the Fourier coefficients:
\begin{equation}\label{pseudo}
\widehat{(\Op(a)f)}(\xi)=\sum_{\zeta\in\Z^2,k\in\Z}\hat{a}(\zeta,k)
\left(\frac{\tilde{\xi}-\tilde{\zeta}}{|\tilde{\xi}-\tilde{\zeta}|}\right)^k 
\hat{f}(\xi-\zeta), \quad \xi\neq\zeta
\end{equation}
(recall that $\tilde{\xi}:=\xi_1+\i\xi_2$ and that the Fourier
coefficients $\hat{a}(\zeta,k)$ decay rapidly) and
\begin{equation}
\widehat{(\Op(a)f)}(\zeta)=\sum_{\zeta\in\Z^2,k\in\Z}\hat{a}(\zeta,k)\hat{f}(0).
\end{equation}
In terms of the Fourier coefficients the matrix elements of $\Op(a)$
can be written as 
\begin{equation}
\left\langle \Op(a)f,f \right\rangle=\sum_{\xi\in\Z^2} \widehat{(\Op(a)f)}(\xi)\overline{\hat{f}(\xi)}
\end{equation}
In particular, for the observable 
$e_{\zeta,k}(x,\phi)=e^{\i\left\langle
    \zeta,x 
  \right\rangle+\i k\phi}$, we have
\begin{equation}
\left\langle \Op(e_{\zeta,k})f,f
\right\rangle=\sum_{\xi\in\Z^2\setminus\{\zeta\}}
\left(\frac{\tilde{\xi}-\tilde{\zeta}}{|\tilde{\xi}-\tilde{\zeta}|}\right)^k\overline{\hat{f}(\xi)}\hat{f}(\xi-\zeta)+\overline{\hat{f}(\zeta)}\hat{f}(0). 
\end{equation}

\subsection{Mixed modes}
If $\zeta\neq0$, we have the bound
\begin{equation}\label{shift bound}
|\left\langle \Op(e_{\zeta,k})f,f \right\rangle|\leq\sum_{\xi\in\Z^2} |\hat{f}(\xi)||\hat{f}(\xi-\zeta)|.
\end{equation}
In the case $f=g_\lambda=G_\lambda/\|G_\lambda\|_2$ we have the $L^2$-expansion $$G_\lambda(x,x_0)=\frac{1}{4\pi^2}\sum_{\xi\in\Z^2} c(\xi)e^{\i\left\langle x,\xi \right\rangle}$$
where $c(\xi)=\frac{1}{|\xi|^2-\lambda}$.
We obtain
\begin{equation}
|\left\langle \Op(e_{\zeta,k})g_\lambda,g_\lambda \right\rangle|\leq
\frac{\sum_{\xi\in\Z^2} |c(\xi)||c(\xi-\zeta)|}{\sum_{\xi\in \Z^2} |c(\xi)|^2}.
\end{equation}
In \cite{RU} it was proved that there exists a full density
subsequence $S'\subset S$ such that for any nonzero lattice vector
$\zeta\in\Z^2$ the matrix elements of $\Op(e_{\zeta,k})$ vanish as
$n\to\infty$ along $S'$. The following result was obtained.

\begin{thm}\label{mixed}{\bf (Rudnick-U., 2012)}
Let $\Lambda$ be an increasing sequence which interlaces with $S$. Denote by
$\lambda_m$ the largest element of $\Lambda$ which is smaller than $m\in S$.
There exists a subsequence $S'\subset S$ of full density such that for
any $\zeta\in\Z^2$, $\zeta\neq0$, $k\in\Z$ 
\begin{equation}
\lim_{\substack{n\to\infty \\ n\in S'}}\left\langle \Op(e_{\zeta,k})g_{\lambda_n},g_{\lambda_n}\right\rangle=0.
\end{equation}
\end{thm}

\begin{remark}
  The above result is only stated for the weak coupling
  quantization in \cite{RU}.  However, the proof in fact works for any
  interlacing sequence, in particular for the strong coupling
  quantization.

To see this, we briefly recall the key steps of the proof
in \cite{RU}. 
The first step is to show that the Green's functions $G_\lambda$ can
be approximated by truncated Green's functions $G_{\lambda,L}$, where
$L=\lambda^\delta$ for a specific choice of $\delta>0$ and the
truncation drops all lattice vectors $\xi$ outside an
annulus $A(\lambda,L)=\{\xi\in\Z^2 \mid ||\xi|^2-\lambda|\leq L\}.$
The subsequence $S'\subset S$ is chosen in such a way to ensure that
the lattice points inside the annulus $A(\lambda,L)$ are sufficiently
well-spaced; this then implies that $c(\xi-\zeta) \ll 1/L$ for $\zeta
\in \Z^2$ fixed and $\xi \in A(\lambda,L)$.
A second condition requires that the neighboring Laplace
eigenvalues are not too far apart in order for the lower bound
$\fnorm{G_{\lambda}}_{2} \gg 1/\lambda^{o(1)}$ to hold.
These two key properties only depend on the arithmetic properties of
the neighboring Laplace eigenvalues, and not on the location of the
new eigenvalue itself.
  
\end{remark}

\subsection{Pure momentum modes}
Let us consider the case $\zeta=0$. We rewrite the matrix elements
as (cf. eq. \eqref{pseudo}) 
\begin{equation}\label{pure momentum matrix element}
\begin{split}
\left\langle
  \Op(e_{0,k})g_\lambda,g_\lambda\right\rangle=&\frac{\sum_{\xi\in\Z^2\setminus\{0\}}
  (\tilde{\xi}/|\tilde{\xi}|)^k|c(\xi)|^2+|c(0)|^2}{\sum_{\xi\in\Z^2}|c(\xi)|^2}\\ 
=&\frac{\frac{1}{\lambda^2}+\sum_{n\in S\setminus\{0\}}\frac{w_k(n)}{(n-\lambda)^2}}{\frac{1}{\lambda^2}+\sum_{n\in S\setminus\{0\}}\frac{r_2(n)}{(n-\lambda)^2}}
\end{split}
\end{equation}
where $w_{k}(n)$, for $n\in S$, is a certain exponential sum defined
as follows: with
$$
\Lambda_{n} := \{ z = x+iy \in \Z[i] : |z|^{2} = n  \},
$$
denoting the set Gaussian integers of norm $n$ (we can interpret these as
lattice points lying on a circle of radius $\sqrt{n}$), we define
\begin{equation}
w_k(n):= 
\sum_{z \in \Lambda_n } \left(\frac{z}{|z|}\right)^k. 
\end{equation}

\section{Pure momentum observables on the square torus}\label{PureModes}
\label{sec:pure-moment-observ}

We begin by introducing some convenient notation.
%\subsection{Notation}
Given a set $S \subset \Z$, let $S(x) := S \cap [1,x]$.  We say that a
subset $S_{1} \subset S$ is of {\em full density} if
$|S_{1}(x)| = (1+o(1))\cdot |S(x)|$ as $x \to \infty$.
In what follows, $S$ will always denote the set of integers that can
be represented as sums of two integer squares.

% Given an integer $n \geq 0$, recall that $r_{2}(n)$ denote the number
% of ways that $n$ can be written as a sum of two integers squares.
% % i.e.,
% % $$
% % r_{2}(n) = |\{ (x,y) \in \Z^2 : x^{2} +y^{2} = n \}|.
% % $$
% Moreover, given $n$ such that $r_{2}(n)>0$, let
% $$
% \Lambda_{n} := \{ z = x+iy \in \Z[i] : |z|^{2} = n  \},
% $$
% and for $k \in \Z$, define a Weyl type exponential sum
% $$
% w_{k}(n) := \sum_{z \in \Lambda_{n}} (z/|z|)^{k}.
% $$

% Since $w_{k}(n)=0$ for
% all $n>0$ if $4 \nmid k$ by the previously mentioned symmetry reason,
% we assume that $k$ is a 
% nonzero integer such that $4|k$.  

For $k \neq 0$, we can now construct a full density subsequence
$S'_k\subset S$ such that $\left\langle
  \Op(e_{0,k})g_{\lambda_n},g_{\lambda_n}\right\rangle\to0$ as
$\lambda_n\to\infty$ along $n\in S'_k$.  (Recall that $\lambda_n$
denotes the perturbed eigenvalue associated with the Laplace
eigenvalue $n\in S$.)
\begin{prop}\label{momentumbound}
For a given integer $k\neq0$, there exists a subsequence
$S'_k\subset S$, of full density, such that for $n\in S'_k$ 
\begin{equation}
|\left\langle \Op(e_{0,k})g_{\lambda_n},g_{\lambda_n}\right\rangle| \ll (\log\lambda_n)^{1/4-\log 2/2+o(1)}.
\end{equation}
\end{prop}
\noindent
We note that $1/4-\log 2/2 = -0.09657\cdots < 0$.

\subsection{Preliminary Results}
\label{sec:preliminary-results}
Before we can give the proof of Proposition \ref{momentumbound} we
state a number of necessary results whose proof can be found in the
number theory literature, or in Section~\ref{sec:numb-theor-backgr}.
We first recall Rieger's bound on
pair correlation type sums for integers that are sums of two squares.
%%  \begin{thm}
%%  \label{thm:pair-corr-bound}
%%    Let $f(n)$ denote the characteristic function of $S$, the set
%%    integers representable as sums of two integer squares.  If $k \neq
%%    0$ and $|k| \ll x$, then
%%  $$
%%  \sum_{n \leq x} f(n) f(n+k) \ll \frac{c(k)x }{\log x}
%%  $$
%%  for some $c(k)$ satisfying $c(k) \ll k/\phi(k)$.  In particular, the
%%  bound
%%  \begin{equation}
%%    \label{eq:ck-bounded-on-average}
%%  \sum_{0<|k| \leq T} c(k) \ll T.
%%  \end{equation}
%%  holds.\fixme{Rieger needs k>0.}
%%  \end{thm}

\begin{thm}[\cite{rieger-sums-of-square-twins}, Satz~2]
\label{thm:rieger-satz-2}  
\label{thm:pair-corr-bound}
Let $f(n)$ denote the characteristic function of $S$, the set integers
representable as sums of two integer squares.  %If $k > 0$ then
If $0 < |h| \ll x$ then
$$
\sum_{n \leq x} f(n) f(n+h) \ll \frac{c(h)x }{\log x}
$$
where
$
c(h) := \prod_{\substack{ p|h \\ p \equiv 3 \mod 4} } (1 + 1/p).
$
\end{thm}
\begin{remark}
  Rieger's result is stated for $h>0$ and summing over $n \leq x+h$,
  but since $f(n) = 0$ for $n < 0$ and we assume $|h|\ll x$, the above
  formulation follows immediately (albeit possibly with a worse
  absolute constant.)  
Moreover, since $c(h) \leq \sum_{d|h} 1/d$, it easily follows that
\begin{equation}
  \label{eq:ck-bounded-on-average}
\sum_{0<|h| \leq T} c(h)  \ll T, \quad \text{ $T \to \infty$}.
\end{equation}
\end{remark}
We shall also need to recall a fundamental fact about the size of
$S(x)$.
\begin{thm}[Landau, see
  \cite{Landau2}, \S183.]
\label{thm:landau}
There exists $c>0$ such that 
\be
|S(x)| = \frac{c \cdot x}{\sqrt{\log x}} (1+O( 1/\log x ))
\ee
as $x \to \infty$.
\end{thm}

Given $n \in S$, let $\omega_1(n)$ denote the number of prime divisors
of $n$ that are congruent to one modulo four, i.e.,
$$
\omega_1(n) := \sum_{p |n, p \equiv 1 \mod 4} 1
$$
We shall use Erd\"os-Kac type techniques to prove (see
Section~\ref{Erdos Kac}) the following structure result about the
factorizations of ``typical'' integers in the set $S$.
\begin{prop}
\label{prop:erdos-kac-moment-estimates}
We have
\be
\frac{1}{|S(x)|}  \sum_{n \in S(x)} \omega_{1}(n) 
= \frac{1}{2} \log \log x + O(\log \log \log x)
\ee
and
\be
\frac{1}{|S(x)|}  \sum_{n \in S(x)} \omega_{1}(n)^{2}
= \frac{1}{4} (\log \log x)^{2} +
O( (\log \log x) \cdot \log \log \log x  )
\ee
\end{prop}
This, together with Chebychev's inequality, immediately gives the
following normal order result on $\omega_{1}(n)$.
\begin{cor} 
\label{cor:r2-log-normal-order}
Fix $\epsilon>0$.  Then, as $x \to \infty$, 
\be
| \{n \in S(x) : |\omega_{1}(n) - \frac{1}{2} \log \log n| <
 (\log \log n )^{1/2+\epsilon} \}|=
|S(x)| \cdot (1+o_{\epsilon}(1)).
\ee
\end{cor}
From the corollary, we deduce (see Section~\ref{Erdos Kac} for
details) the following weak analog of a normal order result for
$r_{2}(n)$.
\begin{cor}
\label{cor:r2-normal-order}
As $x \to \infty$, 
\be
| \{n \in S(x) : r_{2}(n) = (\log n)^{(\log 2)/2 \pm o(1)} \}| =
|S(x)| \cdot (1+o(1)).
\ee
\end{cor}
We shall also need the following $L^{2}$-bound on the exponential sums
$w_{k}(n)$, 
\begin{prop}
\label{prop:w-k-l2-bound}
If $k \neq 0$, then
\begin{equation}
  \label{eq:1}
\sum_{n \in{ S(x)}}  |w_k(n)|^{2} \ll_{k}  x
\end{equation}
In particular, by Chebychev's inequality, the number of $n \in S(x)$
for which $|w_{k}(n)|>T$ is at most $x/T^{2}$, and we find that
$|w_{k}(n)| \leq (\log n)^{1/4+\epsilon}$ holds for almost all $n \in
S(x)$.
\end{prop}
The result readily follows from a Halberstam-Richert type inequality,
see section \ref{Mean Values} for more details.

\subsection{Proof of Proposition \ref{momentumbound}}

We begin by noting that since the set $\Lambda_{n}$ is invariant under
multiplication by $i$, $w_{k}(n) = 0$ unless $4|k$.  Hence the case $k
\not \equiv 0 \mod 4$ is essentially trivial on recalling \eqref{pure
  momentum matrix element}.  Thus, in what follows we will always
assume that $4|k$, and $k \neq 0$.

We next introduce some further notation. 
Given $m \in S$, let $m_{+}, m_{-} \in S$ denote the nearest neighbor
(in $S$) to the right, respectively left, and similarly, let
$m_{++},m_{--}$ denote the second nearest neighbors to the right,
respectively left.

Define $S_{1} \subset S$ by successively removing a zero density
subset of elements for which the following properties do not
hold. Namely, let $S_{1}$ consist of those $m \in S$ for which the
following properties, as $m \to \infty$, all hold:
\begin{enumerate}
\item \label{item:mult-near-mean} Multiplicities are near their
  (logarithmic) normal order in the following sense:
$$r_{2}(m) = (\log  m)^{(\log 2)/2 \pm o(1)}, \quad
r_{2}(m_{-}) = (\log  m)^{(\log 2)/2 \pm o(1)}.
$$
\item  \label{item:almost-square-root-cancellation}
There is nearly square root cancellation in  exponential sums:
$$
|w_{k}(m)|   \leq (\log m)^{1/4+o(1)}, \quad  |w_{k}(m_{-})|
  \leq (\log m)^{1/4+o(1)}.
$$
\item  \label{item:no-near-nbrs}
There are no near neighbors: $m_{+}-m \geq (\log m)^{1/2-o(1)}$, and
  $m-m_{-} \geq (\log m)^{1/2-o(1)}$.
\item  \label{item:no-near-second-nbrs}
There are  no near second neighbors: $m_{++}-m_{+} \geq (\log m)^{1/2-o(1)}$, and
  $m_{-}-m_{--} \geq (\log m)^{1/2-o(1)}$.
\item \label{item:no-far-nbrs} Neighbors are not too far away:
  $m_{+}-m \leq (\log m)^{1/2+o(1)}$, $m-m_{-} \leq (\log
  m)^{1/2+o(1)}$, and $m_{-}-m_{--} \leq (\log m)^{1/2+o(1)}$.
\item \label{item:not-many-close-nbrs}
There are not too many ``close'' neighbors in the following sense: for 
$T \ll m$,
$$
|\{ n \in S: |n-m| \leq T \}| \ll \frac{T (\log T)^{2}}{(\log m)^{1/2-o(1)}}
$$

\item \label{item:remove-w-bad} For $W \in [(\log m)^{1/4} \cdot (\log \log
  m)^{2}, (\log m)^{2}]$ there are 
$$
\gg W^{2}/( (\log m)^{1/2} (\log \log m) (\log W)^{2})
$$ 
elements in $S$ that lie between $m$ and $n$ if $|w_{k}(n)| \geq W$.

\item \label{item:remove-w-really-bad} For $W \geq (\log m)^{2}$ there
  are
$
\gg  W^{3/2}/\log W
$ 
elements in $S$ that lie between $m$ and $n$ if $|w_{k}(n)| \geq W$.

\item \label{item:crucial-sum-bound} For $\epsilon>0$ and 
  $G \in [2, m^{1-\epsilon}]$,
$$
H_{G}(m) := \sum_{\substack{n \in S, n \neq m\\ |m-n| \geq G} }
\frac{1}{|m-n|^{2}}
\ll_{\epsilon} \frac{(\log G)^{2}}{G (\log m)^{1/2-o(1)}}.
$$

\end{enumerate}
\begin{remark}
We tacitly assume that $o(1)$ is chosen so that $(\log m)^{o(1)} \to
\infty$ as $m \in S$ tends to infinity; we will also use the
convention that the sign of $o(1)$ is important, in particular 
$(\log m)^{-o(1)} \to 0$.
\end{remark}

We defer the proof that $S_{1}$ has full  density inside $S$
to Section~\ref{sec:proof-that-s_1}.

\begin{remark}
\label{rem:same-size}
Here, and what follows we will without comment make use of the fact
that $\log m \sim \log m_{-} \sim \log m_{+} \sim \log m_{++}$ etc.
To see this, any crude bound on $|m_{+}-m|$, $|m_{++}-m|$ etc
suffices, e.g. the trivial bound $|m_{+}-m| \ll m^{1/2}$ which follows
from bounding the distance to the nearest square.
  Moreover, we 
  also use the fact that for almost all $m \in S(x)$, e.g. $m \in
  [x/\log x,x]$, we have $\log m \sim \log x$.
\end{remark}

The following will be used to show that the numerator in \eqref{pure
  momentum matrix element} 
is essentially given by two terms.
\begin{lem}
\label{lem:claim-lemma}
If $m \in S_{1}$, then
\be
%1/\lambda_m^{2} + 
\sum_{n \in S, n \neq m, m_{-}}
%\frac{|w_{k}(n)|}{|\lambda_{m}-n|^{2}}
\frac{|w_{k}(n)|}{|m-n|^{2}}
\ll
\frac{1}{(\log m)^{3/4-o(1)}}
\ee
\end{lem}

\begin{proof}
%\subsection{Proof of Lemma~\ref{lem:claim-lemma}}
Fix $m \in S_{1}(x)$.  To simplify the notation, let $L = \log m$.  
To bound the sum
$$
\sum_{n \in S, n \neq m, m_{-}}
%\frac{|w_{k}(n)|}{|\lambda_{m}-n|^{2}}
\frac{|w_{k}(n)|}{|m-n|^{2}}
$$
we split it into parts according to the size of $|w_{k}(n)|$.

{\em Small $|w_{k}(n)|$}: $|w_{k}(n)| \leq L^{1/4} (\log L)^{2}$.
Since $m \in S_{1}$, its nearest neighbors,
by property~\ref{item:no-near-nbrs} 
%by properties~\ref{item:no-near-nbrs} and
%\ref{item:no-near-second-nbrs}, 
are of distance at least $ 
L^{1/2-o(1)}$ away from $m$.  Thus, the contribution from $n$ for
which $|w_{k}(n)| \leq L^{1/4} (\log L)^{2}$ is, by property
\ref{item:crucial-sum-bound},
$$
\ll \sum_{\substack{n \in S \\|n-m| \geq L^{1/2-o(1)}}}
\frac{ L^{1/4} (\log L)^{2}}{|m-n|^{2}} =
L^{1/4} \cdot (\log L)^{2} \cdot H_{L^{1/2-o(1)}}(m)
$$
$$ \ll
\frac{L^{1/4} (\log L)^{4}}{L^{1/2-o(1)}
  L^{1/2-o(1)}} = \frac{1}{L^{3/4 -o(1)}}.
$$

{\em Medium $|w_{k}(n)|$}: $|w_{k}(n)| \in [L^{1/4} (\log
L)^{2},L^{2}]$.
For terms in the sum for which $n \geq 2m$, we use the crude bound
$|w_{k}(n)| \leq r_{2}(n) \ll \sqrt{n}$ and find that the total
contribution is $\ll \sum_{n \geq m} n^{-3/2} \ll m^{-1/2}$, and hence
it is enough to consider terms for which $n < 2m$.

Let $W_{i} = 2^{i} L^{1/4} (\log L)^{2}$ for integer $i \geq 0$ such
that $2^{i} L^{1/4} (\log L)^{2} \leq L^{2}$ and consider $n$ such
that $|w_{k}(n)| \in [W_{i}, W_{i+1}]$.  By
property~\ref{item:remove-w-bad}, the number of elements in $S$ between $n$
and $m$ is $\gg W_{i}^{2}/(L^{1/2+o(1)} (\log W_{i})^{2})$. 
Thus, using the bound on the number of close neighbors, i.e., take $T
= |n-m|$ in property~\ref{item:not-many-close-nbrs} (note that $T \ll
m$ when $n < 2m$), we must have $\frac{T (\log T)^{2}}{(\log
  m)^{1/2-o(1)}} \gg W_{i}^{2}/(L^{1/2+o(1)} (\log W_{i})^{2})$),
which implies that $|n-m|\gg W_{i}^{2-o(1)}$.
Thus, by property~\ref{item:crucial-sum-bound} (take $G =
W_{i}^{2-o(1)}$), 
\begin{multline*}
\sum_{\substack{n \in S \\ |w_{k}(n)| \in [W_{i},W_{i+1}]}}
\frac{|w_{k}(n)|}{|n-m|^{2}} \ll
 \sum_{n \in S : |m-n| \gg W_{i}^{2-o(1)}}
\frac{W_{i}}{|n-m|^{2}} \ll
\frac{W_{i} (\log W_{i})^{2}}{L^{1/2-o(1)}W_{i}^{2-o(1)}}
\\=
\frac{1}{L^{1/2-o(1)} \cdot W_{i}^{1-o(1)}}
=
\frac{1}{L^{1/2-o(1)}\cdot (2^{i}L^{1/4}(\log L)^{2})^{1-o(1)}}
\ll
\frac{1}{L^{3/4-o(1)} (3/2)^{i}}.
\end{multline*}
Summing over relevant $i \geq 0$, we find that the total contribution
is 
$
\ll
\frac{1}{L^{3/4-o(1)}}.
$

{\em Large $|w_{k}(n)|$}: $|w_k(n)| \geq L^{2}$.  Let $W_{i} = 2^{i}
L^{2}$ and consider $n$ such that $|w_{k}(n)| \in [W_{i}, 2W_{i}]$.
By property \ref{item:remove-w-really-bad}, we must then have $|n-m|
\gg W_{i}^{3/2}/\log W_{i}$, and hence the contribution is (using the
bound $\sum_{k \geq A} 1/k^{2} \ll 1/A$)
\begin{equation*}
\begin{split}
\ll  \sum_{\substack{n \in S \\ |n-m|\gg W_{i}^{3/2}/\log W_{i}}}
\frac{W_{i}}{|n-m|^{2}} 
\ll
\frac{W_{i} \cdot \log W_{i}}{W_{i}^{3/2}} 
\ll \frac{1}{W_{i}^{1/2-o(1)}}
= \frac{1}{(2^{i}L^{2})^{1/2-o(1)}}.
\end{split}
\end{equation*}
Summing over  $i \geq 0$, the total contribution is
$$
\ll \frac{1}{L^{1-o(1)}} \sum_{i \geq 0} 2^{-(1/2-o(1))i} \ll 
\frac{1}{L^{1-o(1)}}.
$$

\end{proof}

\begin{proof}[Proof of Proposition~\ref{momentumbound} using
  Lemma~\ref{lem:claim-lemma}]
Recalling that $m_{-} < \lambda_{m} < m$, we note that that
$|\lambda_{m}-n| \geq |m-n|$ if $n>m$. 
Moreover, for $n \leq m_{--}$ the minimum of $|\lambda_{m}-n|/|m-n| =
1 - \frac{m-\lambda_{m}}{m-n}$
(as $n \leq m_{--}$ ranges over elements in $S$)
is attained for $n=m_{--}$, and consequently
$$
|\lambda_{m}-n|/|m-n| \geq |\lambda_{m}-m_{--}|/|m-m_{--}| 
\geq 
|m_{-}-m_{--}|/|m-m_{--}|
$$
which, by properties %\ref{item:no-near-nbrs}, (prop 3 not needed.)
\ref{item:no-near-second-nbrs}, and \ref{item:no-far-nbrs}, is $\gg
(\log m)^{1/2-o(1)}/(\log m)^{1/2+o(1)} = 1/(\log m)^{o(1)}$.  
Hence $|\lambda_{m}-n| \gg (\log m)^{-o(1)} |m-n|$ holds for $n \neq
m, m_{-}$, and thus
$$
\sum_{n \in S, n \neq m, m_{-}}
\frac{|w_{k}(n)|}{|\lambda_{m}-n|^{2}}
\leq
(\log m)^{o(1)} \cdot \sum_{n \in S, n \neq m, m_{-}}
%\frac{|w_{k}(n)|}{|\lambda_{m}-n|^{2}}
\frac{|w_{k}(n)|}{|m-n|^{2}}
$$

Let $M =
\min( |\lambda_{m}-m|^{2}, |\lambda_{m}-m_{-}|^{2})$.  Trivially
$\lambda_m \gg m^{1/2}$, and by property~\ref{item:mult-near-mean},
Lemma~\ref{lem:claim-lemma} implies that
\begin{multline*}
\frac{1/\lambda_m^{2} + \sum_{n \in S}
  \frac{|w_{k}(n)|}{|\lambda_{m}-n|^{2}} }
{1/\lambda_m^{2} + \sum_{n \in S}
  \frac{r_{2}(n)}{|\lambda_{m}-n|^{2}} } 
\ll
\frac{ O(1/m)+ (|w_{k}(m)|+|w_{k}(m_{-})|)/M + \frac{1}{(\log
    m)^{3/4-o(1)}}   }
{\frac{O(1/m)+ (\log m)^{(\log 2)/2-o(1)}}{M}}
\end{multline*}
which, by property~\ref{item:almost-square-root-cancellation}, is
\begin{equation}
  \label{eq:cant-name-this}
\ll
\frac{(\log m)^{1/4+o(1)}+\frac{M }{(\log m)^{3/4-o(1)}}}
{(\log m)^{(\log 2)/2-o(1)}}.
\end{equation}
Recalling that $\lambda_{m}\in [m_{-},m]$,
property~\ref{item:no-far-nbrs} implies that $M \ll (\log m)^{1+o(1)}$
and we thus find that (\ref{eq:cant-name-this}) is
$$
\ll
\frac{(\log m)^{1/4+o(1)}+\frac{(\log m)^{1+o(1)}}{(\log m)^{3/4-o(1)}}}
{(\log m)^{(\log 2)/2-o(1)}}
= 
\frac{1}{(\log m)^{(\log 2)/2-1/4-o(1)}} = o(1)
$$
as $m \to \infty$, since $(\log 2)/2-1/4 = 0.09657\cdots$.  Recalling
that  $\log m \gg \log \lambda_{m}$ (cf. Remark~\ref{rem:same-size})
the proof is concluded.

\end{proof}

\subsection{Proof that $S_1$ has full  density}
\label{sec:proof-that-s_1}

\subsubsection{Property (\ref{item:mult-near-mean})}
That $r_{2}(m) = (\log m)^{(\log 2)/2+o(1)}$ holds for almost all $m \in S$
follows from corollary~\ref{cor:r2-normal-order}.
To ensure that $r_{2}(m_{-}) = (\log m)^{(\log 2)/2+o(1)}$ also holds,
we remove the right neighbor of those $m$ for which $r_{2}(m) =
(\log m)^{(\log 2)/2+o(1)}$ is not true; this removes another zero
density set.  
(By Remark~\ref{rem:same-size}, $\log m_{+} = (1+o(1)) \log m$.)

\subsubsection{Property  (\ref{item:almost-square-root-cancellation})}
By Proposition~\ref{prop:w-k-l2-bound}
$$
\sum_{n \in S(x) } |w_{k}(n)|^{2} \ll  x
$$
and Chebychev's inequality, together with $|S(x)| \sim cx/\sqrt{\log
  x}$, then gives that $|w_k(m)| \leq (\log
m)^{1/4+o(1)}$ holds for almost all $m \in S(x)$.  Removing right
neighbors, as in the proof of Property~(\ref{item:mult-near-mean}),
the same holds for $|w_{k}(m_{-})|$.

\subsubsection{Property  (\ref{item:no-near-nbrs})}
Let $f$ denote the characteristic function of $S$. By
Theorem~\ref{thm:pair-corr-bound}, 
$$
\sum_{m \leq x} \sum_{h : 0 <|h| \leq (\log m)^{1/2-o(1)} } f(m) f(m+h)
\ll
\frac{x}{\log x}
\sum_{h : 0< |h| \leq (\log x)^{1/2-o(1)} } c(h)
$$
$$
\ll
\frac{x }{\log x} \cdot (\log x)^{1/2-o(1)}.
$$
Thus, by Chebychev's inequality, 
$$
\sum_{h : 0<|h| \leq (\log m)^{1/2-o(1)} } f(m) f(m+h) < 1
$$
holds for almost all $m$ in $S(x)$.  Consequently, almost all $m \in
S$ have no nearby neighbors.

\subsubsection{Property  (\ref{item:no-near-second-nbrs})}
We use the same proof as the one used for showing that 
Property  (\ref{item:no-near-nbrs}) holds.

\subsubsection{Property  (\ref{item:no-far-nbrs})}
Let $n_{1} < n_{2} \ldots < n_I \leq x$ denote ordered representatives
of the elements in $S(x)$, and let $s_{i} = n_{i+1}-n_{i}$.  Since
$\sum_{i < I} s_{i} \leq x$, Chebychev's inequality implies that
$s_{i} \leq (\log n_{i})^{1/2+o(1)}$ holds for almost all $n_{i} \in
S$; consequently $n_{+}-n \leq (\log n_{i})^{1/2+o(1)}$ for almost all
$n \in S$.

A similar argument shows that $n_{i}-n_{i-2}  \leq (\log
n_{i})^{1/2+o(1)}$ also holds for almost all $n_{i}$. Hence  both
$n-n_{-} \leq (\log n_{i})^{1/2+o(1)}$ and $n_{-}-n_{--} \leq (\log
n_{i})^{1/2+o(1)}$ holds for almost all $n \in S$.

\subsubsection{Property  (\ref{item:not-many-close-nbrs})}
The argument is similar to the one used to prove property
\ref{item:no-near-nbrs}: again let $f$ denote the characteristic
function of $S$. Then, as $T \ll m \leq x$,
Theorem~\ref{thm:pair-corr-bound} gives that
$$
\sum_{m \leq x} \sum_{h : 0 <|h| \leq T} f(m) f(m+h)
\ll
\frac{x}{\log x}
\sum_{h : 0< |h| \leq T } c(h)
\ll
\frac{xT}{\log x}
$$
and Chebychev's inequality implies that
$$
\sum_{h : 0<|h| \leq T } f(m) f(m+h) \geq
\frac{T (\log T)^{2}}{(\log x)^{1/2-o(1)}}
$$
holds for at most $\frac{x}{(\log x)^{1/2+o(1)} (\log T)^{2}}$
exceptional elements $m \in S(x)$.  Taking $T_{i}=2^{i}$, removing the
exceptional elements, and summing over $i$ we find that we have removed
$$
\ll \frac{x}{(\log x)^{1/2+o(1)}} \sum_{i \geq 0} 1/i^{2} = o(|S(x)|)
$$
elements.

Thus, property~\ref{item:not-many-close-nbrs} holds for $T$ being a
power of two.  To see that it holds for all $T \ll m$, take $i$ to be
the smallest integer such that $T_{i} = 2^{i}\geq T$ and note that $T
\in [T_{i}/2, T_{i}]$. 

\subsubsection{Property  (\ref{item:remove-w-bad})}

Given $n \in S$ such that $$w = |w_{k}(n)| \in [(\log n)^{1/4}(\log
\log n), (\log n)^{2}],$$ remove $2\cdot w^{2}/( (\log n)^{1/2} (\log
\log n) (\log w)^{2})$ neighbors to the left of $n$, and $2 \cdot
w^{2}/( (\log n)^{1/2} (\log \log n) (\log w)^{2})$ neighbors to the
right, and let $R_{n}$ denote the set of such removed elements.

Fix $x$ and consider the number of removed elements in $[1,x]$.  We
claim that if $l \in S(x)$ has been removed, then $l \in R_{n}$ for
some $n \leq 2x$.  To show this, we note that given an integer $t$, we
can always find $l_{1},l_{2} \in S$ such that $l_{1} < t < l_{2}$, and
$l_{2}-l_{1} \ll \sqrt{t}$ (just take nearby squares), and since
$|w_{k}(n)|\leq r_{2}(n) \leq n^{o(1)}$, any $R_{n}$ will be contained
in an interval of length $\ll n^{1/2+o(1)}$.

Hence it suffices to bound the union of $R_{n}$ for $n \leq 2x$.  The
removed contribution from $n$ for which $n \leq x/(\log x)^{10}$ is at
most $\frac{x \cdot (\log x)^{4}}{(\log x)^{10}} = o(|S(x)|)$ (here we
use the assumption $w \leq (\log n)^{2}$).

On the other hand, for $n \in [x/(\log x)^{10},2x]$, we have $\log n =
(1+o(1))\log x$. 
Let $W_{i} = 2^{i} (\log x)^{1/2}(\log \log x)$, and consider $n \in
S(x)$ such that $|w_{k}(n)| \in [W_{i},2W_{i}]$.  By
Proposition~\ref{prop:w-k-l2-bound} and Chebychev's inequality, the
number of such $n$ 
is $\ll \frac{x}{W_{i}^{2}}$, and the total number of removed elements
is thus
$$
\ll \frac{x}{W_{i}^{2}} \cdot 
\frac{W_{i}^{2}}{ (\log x)^{1/2} (\log \log x) (\log W_{i})^{2}}
\ll
\frac{x}{ (\log x)^{1/2} (\log \log x) \cdot i^{2}}
$$
Summing over $i \geq 0$ we find that the total number of removed
elements is 
$$
\ll  \frac{x}{ (\log x)^{1/2} (\log \log x) \cdot i^{2}} = o(|S(x)|).
$$

\subsubsection{Property  (\ref{item:remove-w-really-bad})}

Arguing as before, if $|w_{k}(n)| \geq (\log n)^{2}$
let $w=|w_{k}(n)|$ and remove
the nearest $2w^{3/2}/\log w$ neighbors to the right and left of $n$;
let $R_{n}$ denote the set of removed neighbors.

Fix $x$ and consider the number of removed elements in $[1,x]$.  
We first note that $|w_{k}(n)| \leq r_{2}(n) \ll n^{1/100}$ holds for
all $n \in S$.
Consequently $R_{n}$, if non-empty, contains at most $n^{3/200}$
neighbors of $n$ which (since $|S(2y)|-|S(y)| \gg y/\sqrt{\log y}$ for
all $y$ by Landau) implies that if $l \in S(x)$ and $l$ belongs to
some $R_{n}$, then $n \leq 2x$.

Consider first the removed contribution coming from $R_{n}$ for which
$n \leq \sqrt{x}$.  Since $|w_{k}(n)| \leq r_{2}(n) \ll n^{1/100}$,
the total contribution is
$$
\ll \sqrt{x} \cdot (x^{1/100})^{3/2} = o(|S(x)|).
$$

If $n \in [\sqrt{x}, 2x]$ and $|w_{k}(n)| \geq (\log n)^{2}$, we have 
$$
|w_{k}(n)| \geq (\log x)^{2}/100.
$$
Define $W_{i}= 2^{i} \cdot (\log x)^{2}/100$ and consider the removed
contribution from $R_{n}$ for which $|w_{k}(n )| \in [W_{i},2W_{i}]$.
By Proposition~\ref{prop:w-k-l2-bound} and Chebychev's inequality, the
number of such $n \in S(2x)$ is
$
\ll \frac{x}{W_{i}^{2}}
$
and the associated removed contribution is
$$
\ll \frac{x \cdot (W_{i}^{3/2}/\log W_{i})}{W_{i}^{2}}
\ll \frac{x}{W_{i}^{1/2} \log W_{i}} 
\ll \frac{x}{(2^{i}( \log x)^{2})^{1/2} } = \frac{x}{2^{i/2} \log x}.
$$

Summing over $i \geq 0$ we find that the total contribution is
$$
\ll \sum_{i \geq 0}\frac{x}{2^{i/2} \log x}  = o(|S(x)|).
$$

\subsubsection{Property  (\ref{item:crucial-sum-bound})}

The final property is an immediate consequence of the following
Lemma.

\begin{lem} If $\epsilon>0$ then for almost all $m \in S(x)$, we have,
  for any $T \in [2, x^{1-\epsilon}]$,
$$
\sum_{\substack{n \in S\\ |n-m| \geq T}} \frac{1}{(m-n)^{2}} \ll
\frac{(\log T)^{2}}{T (\log x)^{1/2-o(1)}}
$$
\end{lem}
\begin{proof}
  We first bound the sum over $n \in S \setminus S(2x)$, i.e., those
  $n$ for which $n \geq 2x$: 
$$
\sum_{\substack{n \in S\\ |n-m| \geq T\\ n \geq 2x}}
\frac{1}{(m-n)^{2}} \leq \sum_{k\geq x} 1/k^{2}
\ll 1/x
= o \left( \frac{(\log T)^{2}}{T (\log x)^{1/2-o(1)}} \right).
$$
(Recall that $m \leq x$ since $m \in S(x)$, and that $T \leq x^{1-\epsilon}$.)

Next we note that
$$
\sum_{\substack{m,n \in S(2x)\\ |n-m| \geq T}} \frac{1}{(m-n)^{2}} 
=
\sum_{k \geq T} \frac{|\{ m,n \in S(2x): |m-n| = k \}|}{k^{2}}
$$
By Theorem~\ref{thm:pair-corr-bound},
$$
|\{ m,n \in S(2x): |m-n| = h \}| \ll \frac{x \cdot c(h)}{\log x}
$$
and, by partial summation and using that $c(h)$ is bounded on average
(cf. (\ref{eq:ck-bounded-on-average}), 
$$
%\sum_{\substack{m \in S(2x)}}
%\sum_{\substack{n \in S(2x)\\ |n-m| \geq T}} \frac{1}{(m-n)^{2}}  =
\sum_{\substack{m,n \in S(2x)\\ |n-m| \geq T}} \frac{1}{(m-n)^{2}} 
\ll
\frac{x}{\log x} \sum_{h \geq T} \frac{c(h) }{h^{2}}
\ll \frac{x}{T \log x}.
$$
By Chebychev's inequality, the number of $m \in S(2x)$ for which 
$\sum_{\substack{n \in S(2x)\\ |n-m| \geq T}} \frac{1}{(m-n)^{2}}  \geq
\frac{(\log T)^{2}}{T (\log x)^{1/2-o(1)}}$ holds is thus
$$
\ll \frac{x}{T \log x} \bigg/ \frac{(\log T)^{2}}{T (\log x)^{1/2-o(1)}} =
\frac{x}{(\log x)^{1/2+o(1)} \cdot (\log T)^{2}}
$$
Taking $T_{i} = 2^{i}$, summing over $i \ll \log x$, and recalling
that $|S(x)| \sim x/\sqrt{\log x}$ we find that the property holds in
the special case of $T$ being a power of two.  The result for general
$T$ follows by taking the largest $i$ such that $T_{i} = 2^{i} \leq T$
and noting that $T_{i} \in [T/2, T]$.
\end{proof}

\section{Proof of Theorem \ref{QE}}
Let $a\in C^\infty(S^*\T^2)$ be a smooth observable with rapidly
decaying Fourier expansion
$$
a(x,\phi)=\sum_{\zeta\in\Z^2 ,k\in\Z}\hat{a}(\zeta,k)e^{\i\left\langle
    \zeta,x \right\rangle+\i   k\phi}.
$$
Since $\Op(e_{\zeta,k})$ is unitary (cf. \eqref{pseudo}) for all
$\zeta,k$, we have $|\left\langle \Op(e_{\zeta,k})g_\lambda,
  g_\lambda\right\rangle| = |\left\langle g_\lambda,
  g_\lambda\right\rangle|$, and the rapid decay of Fourier
coefficients then shows that given $\epsilon>0$, there exists $J$ such
that \be
\begin{split}
|\left\langle(\Op(a)-\Op(P_J))g_\lambda,g_\lambda\right\rangle|
\leq & \sum_{|\zeta|,|k|>J}|\hat{a}(\zeta,k)||\left\langle \Op(e_{\zeta,k})g_\lambda, g_\lambda\right\rangle| 
\leq \epsilon
\end{split}
\ee holds {\em uniformly} in $\lambda$, where $P_J(x,\phi)$ is the
trigonometric polynomial
\begin{equation}\label{trigpoly}
P_J(x,\phi)=\sum_{\substack{\zeta\in\Z^2,k\in\Z\\ |\zeta|,|k|\leq J}}\hat{a}(\zeta,k)e^{\i\left\langle \zeta,x \right\rangle+\i k\phi}
\end{equation}
obtained by truncating the Fourier expansion of $a$.  Hence it is
enough to show that for any fixed $J \geq 1$,
\begin{equation}
\left\langle \Op(P_J)
  g_\lambda,g_\lambda\right\rangle \to
\frac{1}{\Vol(S^*\T^2)}\int_{S^*\T^2} a \,
d\mu 
= \hat{a}(0,0)
\end{equation}
as $\lambda \to \infty$ along a full density subsequence of $S$.

Now, given $J \geq 1$, let
$$\tilde{S}_J:= \bigcap_{|k|\leq J}
\left( S'_k\cap S'  \right)$$ 
where $S'\subset S$ denotes the full density sequence of Theorem
\ref{mixed}.  (Since $S'$ and $S_{k}'$ have full densities for all $k
\neq 0$, so does $\tilde{S}_{J}$ for all $J$.)
It follows from the previous two sections that 
\begin{equation}
\label{polylimit}
\left\langle \Op(P_J)
  g_\lambda,g_\lambda\right\rangle\to\frac{1}{\Vol(S^*\T^2)}\int_{S^*\T^2}P_J
d\mu 
= \hat{a}(0,0)
\end{equation}
as $\lambda\in \tilde{S}_J\to\infty$.

In order to construct the full density sequence of Theorem \ref{QE} we
use a standard diagonalisation argument (see for instance \cite{CdV2})
to extract such a sequence from the list of sequences
$\{\tilde{S}_J\}_{J}$. By construction $\tilde{S}_{J+1}\subset
\tilde{S}_J$. Choose $M_J$ such 
that for all $X>M_J$ 
\be
\frac{\#\{\lambda\in \tilde{S}_J \mid \lambda\leq X\}}{\#\{\lambda\in S \mid \lambda\leq X\}}\geq 1-\frac{1}{2^J}
\ee
and let $S'_\infty$ be such that
$S'_\infty\cap[M_J,M_{J+1}]=\tilde{S}_J\cap[M_J,M_{J+1}]$ for all $J$. Then
$S'_\infty\cap[0,M_{J+1}]$ contains $\tilde{S}_J\cap[0,M_{J+1}]$ and therefore
$S'_\infty$ is of full density in $S$.  Moreover, for any $J\geq 1$, we have 
\begin{equation}
\left\langle \Op(P_J)g_\lambda,g_\lambda \right\rangle 
=
\int_{S^*\T^2}P_J d\mu_\lambda \to
\frac{1}{\Vol(S^*\T^2)}\int_{S^*\T^2}P_J d\mu 
=\hat{a}(0,0)
\end{equation}
as $\lambda\to\infty$ along $S'_\infty$ since  $S'_{\infty} \cap (M_{J+1},\infty)
\subset \tilde{S}_{J} \cap (M_{J+1},\infty)$.

\section{Number theoretic background}
\label{sec:numb-theor-backgr}

\subsection{Bounding mean values of multiplicative functions}
\label{Mean Values}

We recall that $r_{2}(n)/4$ is a {\em multiplicative function}, i.e.,
$r_{2}(mn)/4 = r_{2}(m)/4 \cdot r_{2}(n)/4$ if $(m,n)=1$, and
similarly $w_{k}(n)/4$ is also multiplicative (e.g., see the proof of
Proposition~6 in \cite{fkw-lattice}.)  In particular, both functions
are determined by 
their values at prime powers, and we have
$$
\frac{r_{2}(p^{e})}{4} = 
\begin{cases}
%1 & \text{if $p = 2$,}
 e+1 & \text{if $p \equiv 1 \mod 4$,}
\\ 1 & \text{if $p \equiv 3 \mod 4$ and $e$ is even, or if $p=2$,}
\\ 0 & \text{if $p \equiv 3 \mod 4$ and $e$ is odd.}
\end{cases}
$$
For $p \equiv 1 \mod 4$, define the angle $\theta_{p} \in [0,\pi/4)$
by $\cos \theta_p = x/\sqrt{x^{2}+y^{2}}$, where $x^{2}+y^{2} = p$ for
$x,y \in \Z$ and $0 \leq y \leq x$.  We then have (if $4|k$)

$$
\frac{w_{k}(p^{e})}{4} = 
\begin{cases}
\sum_{l=0}^{e}   e^{i \cdot \theta_{p} \cdot (e-2l)} 
& \text{if $p \equiv 1 \mod 4$,}
\\ 1 & \text{if $p \equiv 3 \mod 4$ and $e$ is even,}
\\ 0 & \text{if $p \equiv 3 \mod 4$ and $e$ is odd,}
\\ \pm 1 & \text{if $p=2$.}
\end{cases}
$$
In particular (again for $4|k$), $w_{k}(2)/4 =  (-1)^{k/4}$, and for odd
primes we have
\begin{equation}
  \label{eq:weyl-sum-on-primes}
\frac{w_{k}(p)}{4} = 
\begin{cases}
2 \cos(k \theta_{p}) & \text{for $p \equiv 1 \mod 4$,}\\
0 &\text{for $p \equiv 3 \mod 4$.}
\end{cases}
\end{equation}

Now, let $f$ be a non-negative multiplicative function such  that
for all prime powers $f(p^{k}) \ll \gamma^{k}$ holds for some $\gamma
< 2$, and 
$$
\sum_{p \leq x } f(p) = \frac{x}{\log x} \cdot (\tau + o(1)),
$$
as $x \to \infty$, for some constant $\tau$.  Satz~1 of Wirsing
\cite{Wirsing} then implies that 
$$
\sum_{n \leq x } f(n)
\ll_{\tau} \frac{x}{\log x} \cdot
\prod_{p \leq x}
\left(
1+\frac{f(p)}{p}+ \frac{f(p^{2})}{p^{2}}+  \ldots
\right).
$$
%(in fact, Wirsing determines the asymptotic, but we shall only need an
%upper bound.)

%\begin{proof}[Proof of Proposition~\ref{prop:w-k-l2-bound}]
\subsubsection{Proof of Proposition~\ref{prop:w-k-l2-bound}}
  For $k$ fixed, define a multiplicative function $f(n) :=
  (|w_{k}(n)|/4)^{2}$ (recall that $|w_{k}(n)|/4$ is 
  multiplicative.)
By \eqref{eq:weyl-sum-on-primes}, we have
$$
f(p) =
\begin{cases}
1 & \text{for $p=2$}, \\
 (2 \cos(k\theta_p))^{2} & \text{for $p \equiv 1 \mod 4$}, \\
 0 & \text{for $p \equiv 3 \mod 4$,}
\end{cases}
$$
and we find that
$$
\sum_{p \leq x} f(p) =
1 + \sum_{p \leq x, p \equiv 1 \mod 4} f(p) =
1 + \sum_{p \leq x, p \equiv 1 \mod 4} (2 \cos( 2 \pi k \theta_{p}))^{2}.
$$
Thus Hecke's result on angular equidistribution of split Gaussian primes
(see \cite{Hecke}) gives that 
\begin{equation}
  \label{eq:prime-sum}
\sum_{p \leq x} f(p)
= \frac{x}{\log x} \cdot 
\left(
\frac{1}{2} \cdot \int_0^{1} (2 \cos(2 \pi k \theta))^{2} \, d \theta +o(1)
\right)
=
\frac{x}{\log x} \cdot  \left( 1 +o(1) \right)
\end{equation}

Hence Wirsing's Satz~1 applies (also note that $f(p^{k}) \ll k^{2}$
for all $p,k$), thus
$$
\sum_{n \leq x} |w_{k}(n)|^{2} \ll 
\sum_{n \leq x} f(n) \ll 
\frac{x}{\log x} \cdot 
\prod_{p \leq x}
\left(
1+\frac{f(p)}{p}+ \frac{f(p^{2})}{p^{2}}+  \ldots
\right).
$$
Now, since $\sum_{k=2}^{\infty} f(p^{k})/p^{k} \leq 
\sum_{k=2}^{\infty} (k+1)^{2}/p^{k} \ll 1/p^{2}$, we find
that 
$$
\prod_{p \leq x}
\left(
1+\frac{f(p)}{p}+ \frac{f(p^{2})}{p^{2}}+  \ldots
\right) \ll
\exp \left( \sum_{p \leq x} \frac{f(p)}{p} \right)
=
\exp \left( \log \log x + O(1) \right),
$$
where the final equality follows from (\ref{eq:prime-sum}) and partial
summation.
Hence
$$
\sum_{n \leq x} |w_{k}(n)|^{2} \ll \frac{x}{\log x} \cdot \exp( \log
\log x + O(1)) \ll x.
$$
% (Recall that the angle $\theta_{p}$ is defined by $\cos \theta_p =
% x/\sqrt{x^{2}+y^{2}}$ where $x^{2}+y^{2} = p$ for $x,y \in \Z$, and
% $\theta_{p} \in [0,\pi/4)$.)

%
%\end{proof}

\subsection{Erd\"os-Kac Theory}\label{Erdos Kac}
Let $\omega(n)$ denote the number of distinct prime factors of an
integer $n$.  The celebrated Erd\"os-Kac theorem assert that the
distribution of $ \left\{ \frac{\omega(n)-\log\log n}{\sqrt{\log\log
      n}} \right\}_{n \leq x} $ is given by the standard normal
distribution as $x\to \infty$; in particular, a typical integer of
size $x$ has about $\log \log x$ prime factors.  We shall need some
analogous, but weaker,  results for elements in $S$.

Recall that given $n\in S$, $\omega_{1}(n)$ denotes the number of
prime factors, congruent to one modulo four, of $n$; i.e.,
with $\sum'_{p}$ denoting the sum over $p \equiv 1 \mod 4$, 
$$
\omega_1(n) := \sum'_{p|n} 1.
$$

%\begin{proof}[Proof of
%Proposition~\ref{prop:erdos-kac-moment-estimates}]
\subsubsection{Proof of Proposition~\ref{prop:erdos-kac-moment-estimates}}

Using that at most four primes $p \geq x^{1/4}$ can
divide an integer $n \leq x$, together with 
$\sum'_{p \leq x^{1/4} }  |\{ n \in S(x) : p|n  \}|=
\sum'_{p \leq x^{1/4} }  |S(x/p)|,$
we find that
\begin{multline*}
\sum_{n \in S(x)} \omega_1(n)  =
\sum_{n \in S(x)} \sum'_{p|n} 1 =
\sum_{n \in S(x)} \left(
\sum'_{p|n, p \leq x^{1/4}} 1  + O(1) \right)
\\=
\sum'_{ p \leq x^{1/4}} |S(x/p)| + O(|S(x)|)
\end{multline*}

By Landau, $|S(x/p)| = \frac{cx}{p \sqrt{\log(x/p)}} \cdot (1+O(1/\log
(x/p))$, and thus, what will be the main term, is given by
$$
\sum'_{ p \leq x^{1/4}} |S(x/p)|  =
\sum'_{ p \leq x^{1/4}} 
\frac{cx}{p \sqrt{\log(x/p)}}
\cdot (1+O(1/\log (x/p))
$$
If $p \in [x^{1/\log\log x}, x^{1/4}]$ then $\log(x/p) \gg \log x$
and thus the contribution from such primes $p$ is 
$$
\ll
\frac{x}{\sqrt{\log x}}
\sum'_{p \in [x^{1/\log\log x}, x^{1/4}]}  1/p
\ll
\frac{x}{\sqrt{\log x}} 
\cdot
\log
\left(
\frac{\log x^{1/4}}{ \log x^{1/\log\log x} }
\right)
\ll
\frac{x \log \log \log x}{\sqrt{\log x}} 
$$
which is of the same order as the claimed error term in the first
assertion of the Proposition.

Now, if $ p \leq x^{1/\log\log x}$ then $\log p \leq \log x/\log
\log x$ and thus
$$
\sqrt{\log(x/p)} = \sqrt{\log x- \log p} = 
\sqrt{ \log x} \left(1 - O
\left( \frac{1 }{\log \log x} 
\right) \right)
$$
Hence, by the analogue of Mertens' theorem for primes in
progressions\footnote{We shall only need that $\sum'_{p \leq x}1/p =
  1/2 \cdot \log \log x + O(1)$, a simple consequence of the prime number
  theorem for arithmetic progressions.}, 
together with $1/\log(x/p) \ll 1/\log x$ for $p \leq x^{1/\log\log
  x}$, we find that
\begin{multline*}
\sum'_{p \leq x^{1/\log\log x}} \frac{cx }{p \sqrt{\log(x/p)}} \cdot
(1+O(1/\log (x/p))
\\=
\frac{cx }{\sqrt{\log x}} \cdot (1+O(1/\log x)) \cdot (1+O(1/\log \log x)))
\sum'_{p \leq x^{1/\log\log x}} 1/p
= \\
\frac{cx }{\sqrt{\log x}} \cdot (1+O(1/\log \log x)))
\cdot
\left( \frac{1}{2} \log \left( \frac{\log x}{\log\log x} \right) +O(1)
\right)
= \\
\frac{cx }{\sqrt{\log x}} \cdot (1+O(1/\log \log x)))
\cdot (\frac{1}{2}\log \log x  + O(\log \log \log x)) 
= \\
\frac{cx }{\sqrt{\log x}} 
( \frac{1}{2} \log \log x + O(\log \log \log x))
\end{multline*}
Dividing by $|S(x)|$ and again using Landau's Theorem, the proof of
the first assertion is concluded.

The variance estimate is similar: since $n \leq x$ can have at most
$4$ prime divisors $p \geq x^{1/4}$, we have
\begin{multline*}
\sum_{n \in S(x)} \omega_1(n)^{2} =
\sum_{n \in S(x)}
\left( \sum'_{p \leq x, p|n} 1 \right)^{2}
=
\sum_{n \in S(x)}
\left( \sum'_{p \leq x^{1/4}, p|n} 1 + O(1) \right)^{2}
\\ =
\sum_{n \in S(x)}
\left(
\left( \sum'_{p \leq x^{1/4}, p|n} 1 \right)^{2}
+ 2 \sum'_{p \leq x^{1/4}, p|n} 1 
+ O(1)
\right)
\end{multline*}
The total contribution from the last two terms in the inner sum is, by
our first assertion (regarding the mean value of $\omega_1(n)$), 
$$
\ll
 \frac{x \log \log x}{\sqrt{\log x}}  + |S(x)|
\ll \frac{x \log \log x}{\sqrt{\log x}}.
$$

With $[a,b] = ab/(a,b)$ denoting the least common multiple of integers
$a,b$, we have
\begin{multline*}
\sum_{n \in S(x)}
\left( \sum'_{p \leq x^{1/4}, p|n} 1 \right)^{2}
=
\sum'_{p_{1},p_{2} \leq x^{1/4}}
|S(x/[p_{1},p_{2}])|
\\=
\sum'_{p_{1},p_{2} \leq x^{1/4}} 
|S(x/p_{1}p_{2})|
+ 
\sum'_{p \leq x^{1/4}} 
|S(x/p)|
-
\sum'_{p \leq x^{1/4}} 
|S(x/p^{2})|
\end{multline*}
The latter two terms are of lower order than the claimed main term ---
the argument used to estimate the mean of $\omega_1(n)$ implies that
$$
\sum'_{p \leq x^{1/4}} 
|S(x/p)|
\ll \frac{x \log \log x}{\sqrt{\log x}}
$$
and, again using Landau, we find that
$$
\sum'_{p \leq x^{1/4}} 
|S(x/p^{2})|
\ll \frac{x}{\sqrt{\log x}} \sum'_{p \leq x^{1/4}} 1/p^{2}
\ll \frac{x}{\sqrt{\log x}}.
$$

As for the double sum over small primes, again by Landau,
\begin{multline}
\label{eq:variance-small-prime-double-sum}
\sum'_{p_{1},p_{2} \leq x^{1/4}} 
|S(x/p_{1}p_{2})|
=
\sum'_{p_{1},p_{2} \leq x^{1/4}} 
\frac{c x}{p_{1}p_{2} \sqrt{\log(x/(p_{1}p_{2}))}}
\left( 1 + O(1/\log(x/(p_{1}p_{2})))\right)
\\= 
\sum'_{p_{1},p_{2} < x^{1/\log \log x}} \ldots
+
2 \cdot \sum'_{\substack{
p_{1} \leq x^{1/\log \log x}\\ p_{2} \in [x^{1/\log \log
      x}, x^{1/4}]}}  \ldots
+
 \sum'_{p_{1}, p_{2} \in [x^{1/\log \log  x}, x^{1/4}]} \ldots
\end{multline}
Again by the analogue of Mertens' Theorem for arithmetic progressions,
$$
 \sum'_{p\in [x^{1/\log \log  x}, x^{1/4}]}  1/p \ll \log \log \log x
$$
and 
$$
\sum'_{p \leq x^{1/\log \log  x}}  1/p = \frac{1}{2} \log \log x
+O(\log \log \log x)
$$
hence the contribution from the latter two sums in
(\ref{eq:variance-small-prime-double-sum}) is
$$
\ll \frac{x}{\sqrt{\log x}} (
(\log \log \log x) \cdot \log \log x +
(\log \log \log x)^{2} )
\ll
\frac{x \cdot (\log \log \log x) \cdot \log \log x  }{\sqrt{\log x}}
$$
which is of the same size as the claimed error term.

Finally, yet again by Landau, and that $\log(x/(p_{1}p_{2})) = \log x
(1+O(1/\log \log x))$ for $p_{1},p_{2} \leq x^{1/\log \log x}$,
we  find that
\begin{multline*}
\sum'_{p_{1},p_{2} < x^{1/\log \log x}}  |S(x/p_{1}p_{2})|
=
\sum'_{p_{1},p_{2} < x^{1/\log \log x}}  
\frac{c x}{p_{1}p_{2} \sqrt{\log( x/(p_{1}p_{2}))}} (1+O(1/\log x)) 
\\=
\sum'_{p_{1},p_{2} < x^{1/\log \log x}}  
\frac{c x}{p_{1}p_{2} \sqrt{\log x}} (1+O(1/\log x)) (1+O(1/\log\log
x)) 
\\=
\frac{cx }{\sqrt{\log x}}
\left(
\sum'_{p < x^{1/\log \log x}} 1/p 
\right)^{2}
(1+O(1/\log \log x))
\end{multline*}
Again by Mertens's for primes in progressions, we find that the main
term equals
$$
\frac{cx }{\sqrt{\log x}}
\left(
\frac{1}{2} \log \log x - O(\log\log \log x)
\right)^{2}
\cdot (1+O(1/\log \log x))
$$
Dividing by $|S(x)|$ and using Landau again, the main term is thus
$$
\frac{1}{4} (\log \log x)^{2} + O( (\log \log x) \cdot \log \log \log x  ).
$$

%\end{proof}

%\begin{proof}[Proof of Corollary~\ref{cor:r2-normal-order}]
\subsubsection{Proof of Corollary~\ref{cor:r2-normal-order}}
\label{sec:proof-coroll-refc}
Define a multiplicative function
$$
f(n) := \frac{r_{2}(n)}{4 \cdot 2^{\omega_1(n)}}
$$
If $p \equiv 1 \mod 4$, then $f(p^{e}) = (e+1)/2$; if $p \equiv 3 \mod
4$ then $f(p^{2e+1})=0$ whereas for even exponents $f(p^{2e})=1$.  Using
Wirsing's Satz~1 again, we find  (recall that $\sum'_{p\leq x}$
denotes the over primes $p \equiv 1 \mod 4$) that
\begin{multline*}
\sum_{n \in S(x)} f(n) = \sum_{n \leq x} f(n)
\ll
\frac{x}{\log x} \exp \left( \sum_{p \leq x} f(p)/p \right)
=
\frac{x}{\log x} \exp \left( \sum'_{\substack{p \leq x 
%\\ p \equiv 1    \mod 4
}} 1/p  + O(1)\right)
\\
\ll
\frac{x}{\log x} \exp \left( \frac{1}{2} \log \log x +O(1) \right)
\ll
\frac{x}{(\log x)^{1/2}} \ll |S(x)|
\end{multline*}
(here we again have used Mertens' Theorem for arithmetic progressions.)
Chebyshev's inequality then implies that the number of $n \in S(x)$
for which $f(n) \geq \log \log \log n$ is $o(|S(x)|)$.  In particular,
we find that 
$$
2^{\omega_1(n)} \leq r_{2}(n)/4 \leq  2^{\omega_1(n)} \cdot \log \log \log n
$$
holds for almost all $n \in S(x)$.  Now, since
Corollary~\ref{cor:r2-log-normal-order} 
%Proposition~\ref{prop:erdos-kac-moment-estimates} 
implies that
$\omega_1(n) = (1/2+o(1)) \log \log n$ for almost all $n \in S(x)$, we
find that
$$
r_{2}(n) = 2^{(1/2+o(1)) \log \log n } = (\log n)^{(\log 2)/2+o(1)}
$$
holds for almost all $n \in S(x)$.
%\end{proof}

% \bibliographystyle{abbrv} 
% \bibliography{mybib,mypapers}
% \end{document}

\end{document}